\documentclass[12pt]{amsart}
\usepackage{amscd,amssymb,latexsym}
\usepackage{fullpage}
\newcommand{\Mdef}[2]{\newcommand{#1}{\relax \ifmmode #2 \else $#2$\fi}}

\newcommand{\supp}{\mathrm{supp}}

\newcommand{\thick}{\mathrm{Thick}}
\newcommand{\loc}{\mathrm{Loc}}

\newcommand{\sm }{\wedge}

\newcommand{\tensor}{\otimes}

\newcommand{\Ext}{\mathrm{Ext}}
\Mdef{\bhom}{\mathbf{\hat{H}om}}

\Mdef{\Mod}{\mathrm{mod}}

\newcommand{\st}{\; | \;}

\newtheorem{thm}{Theorem}[section]
\newtheorem{lemma}[thm]{Lemma}
\newtheorem{prop}[thm]{Proposition}
\newtheorem{cor}[thm]{Corollary}

\theoremstyle{definition}

\newtheorem{defn}[thm]{Definition}
\newtheorem{cond}[thm]{Condition}

\newtheorem{example}[thm]{Example}

\newtheorem{remark}[thm]{Remark}

\newcommand{\qqed}{\qed \\[1ex]}
\renewenvironment{proof}[1][\hspace*{-.8ex}]{\noindent {\bf Proof #1:\;}}{\qqed}

\Mdef{\PH} {\Phi^H}
\Mdef{\PK} {\Phi^K}
\Mdef{\PL} {\Phi^L}
\Mdef{\PT} {\Phi^{\T}}

\Mdef{\ef}{E{\cF}_+}
\Mdef{\etf}{\widetilde{E}{\cF}}
\Mdef{\eg}{E{G}_+}
\Mdef{\etg}{\tilde{E}{G}}
\newcommand{\etp}{\tilde{E}\cP}

\newcommand{\piA}{\pi^{\cA}}

\Mdef{\infl}{\mathrm{inf}}
\Mdef{\defl}{\mathrm{def}}
\Mdef{\res}{\mathrm{res}}
\Mdef{\ind}{\mathrm{ind}}
\Mdef{\coind}{\mathrm{coind}}

\Mdef{\univ}{\mathcal{U}}

\Mdef{\Fp}{\mathbb{F}_p}
\Mdef{\Zpinfty}{\Z /p^{\infty}}
\Mdef{\Zpadic}{\Z_p^{\wedge}}

\newcommand{\bi}{\begin{itemize}}
\newcommand{\be}{\begin{enumerate}}
\newcommand{\bc}{\begin{center}}
\newcommand{\bd}{\begin{description}}
\newcommand{\ei}{\end{itemize}}
\newcommand{\ee}{\end{enumerate}}
\newcommand{\ec}{\end{center}}
\newcommand{\ed}{\end{description}}

\newcommand{\lra}{\longrightarrow}

\newcommand{\Gspectra}{\mbox{$G$-{\bf spectra}}}

\newcommand{\spec}{\mathrm{Spec}}
\newcommand{\spc}{\mathrm{Spc}}

\Mdef{\we}{\mathbf{we}}
\Mdef{\fib}{\mathbf{fib}}
\Mdef{\cof}{\mathbf{cof}}
\Mdef{\BI}{\mathcal{BI}}

\newcommand{\colim}{\mathop{  \mathop{\mathrm {lim}} \limits_\rightarrow} \nolimits}

\newcommand{\hocolim}{\mathop{  \mathop{\mathrm {holim}}\limits_\rightarrow} \nolimits}

\Mdef{\B}{\mathbb{B}}
\Mdef{\C}{\mathbb{C}}
\Mdef{\D}{\mathbb{D}}
\Mdef{\E}{\mathbb{E}}
\Mdef{\T}{\mathbb{T}}
\Mdef{\F}{\mathbb{F}}
\Mdef{\G}{\mathbb{G}}
\Mdef{\I}{\mathbb{I}}
\Mdef{\N}{\mathbb{N}}
\Mdef{\Q}{\mathbb{Q}}
\Mdef{\R}{\mathbb{R}}
\Mdef{\bbS}{\mathbb{S}}
\Mdef{\Z}{\mathbb{Z}}

\Mdef{\bA}{\mathbb{A}}
\Mdef{\bB}{\mathbb{B}}
\Mdef{\bC}{\mathbb{C}}
\Mdef{\bD}{\mathbb{D}}
\Mdef{\bE}{\mathbb{E}}
\Mdef{\bF}{\mathbb{F}}
\Mdef{\bG}{\mathbb{G}}
\Mdef{\bH}{\mathbb{H}}
\Mdef{\bI}{\mathbb{I}}
\Mdef{\bJ}{\mathbb{J}}
\Mdef{\bK}{\mathbb{K}}
\Mdef{\bL}{\mathbb{L}}
\Mdef{\bM}{\mathbb{M}}
\Mdef{\bN}{\mathbb{N}}
\Mdef{\bO}{\mathbb{O}}
\Mdef{\bP}{\mathbb{P}}
\Mdef{\bQ}{\mathbb{Q}}
\Mdef{\bR}{\mathbb{R}}
\Mdef{\bS}{\mathbb{S}}
\Mdef{\bT}{\mathbb{T}}
\Mdef{\bU}{\mathbb{U}}
\Mdef{\bV}{\mathbb{V}}
\Mdef{\bW}{\mathbb{W}}
\Mdef{\bX}{\mathbb{X}}
\Mdef{\bY}{\mathbb{Y}}
\Mdef{\bZ}{\mathbb{Z}}

\Mdef{\cA}{\mathcal{A}}
\Mdef{\cB}{\mathcal{B}}
\Mdef{\cC}{\mathcal{C}}
\Mdef{\mcD}{\mathcal{D}} % Something funny about \cD.
\Mdef{\cE}{\mathcal{E}}
\Mdef{\cF}{\mathcal{F}}
\Mdef{\cG}{\mathcal{G}}
\Mdef{\mcH}{\mathcal{H}} % There's something funny about \cH: it 
\Mdef{\cI}{\mathcal{I}}
\Mdef{\cJ}{\mathcal{J}}
\Mdef{\cK}{\mathcal{K}}
\Mdef{\mcL}{\mathcal{L}}% There's something funny about \cL: it 
\Mdef{\cM}{\mathcal{M}}
\Mdef{\cN}{\mathcal{N}}
\Mdef{\cO}{\mathcal{O}}
\Mdef{\cP}{\mathcal{P}}
\Mdef{\cQ}{\mathcal{Q}}
\Mdef{\mcR}{\mathcal{R}}% There's something funny about \cR: it 
\Mdef{\cS}{\mathcal{S}}
\Mdef{\cT}{\mathcal{T}}
\Mdef{\cU}{\mathcal{U}}
\Mdef{\cV}{\mathcal{V}}
\Mdef{\cW}{\mathcal{W}}
\Mdef{\cX}{\mathcal{X}}
\Mdef{\cY}{\mathcal{Y}}
\Mdef{\cZ}{\mathcal{Z}}

\Mdef{\ca}{\mathcal{a}}
\Mdef{\ct}{\mathcal{t}}

\Mdef{\At}{\tilde{A}}
\Mdef{\Bt}{\tilde{B}}
\Mdef{\Ct}{\tilde{C}}
\Mdef{\Et}{\tilde{E}}
\Mdef{\Ht}{\tilde{H}}
\Mdef{\Kt}{\tilde{K}}
\Mdef{\Lt}{\tilde{L}}
\Mdef{\Mt}{\tilde{M}}
\Mdef{\Nt}{\tilde{N}}
\Mdef{\Pt}{\tilde{P}}

\Mdef{\tA}{\tilde{A}}
\Mdef{\tB}{\tilde{B}}
\Mdef{\tC}{\tilde{C}}
\Mdef{\tE}{\tilde{E}}
\Mdef{\tH}{\tilde{H}}
\Mdef{\tK}{\tilde{K}}
\Mdef{\tL}{\tilde{L}}
\Mdef{\tM}{\tilde{M}}
\Mdef{\tN}{\tilde{N}}
\Mdef{\tP}{\tilde{P}}

\Mdef{\ft}{\tilde{f}}
\Mdef{\xt}{\tilde{x}}
\Mdef{\yt}{\tilde{y}}

\Mdef{\Ab}{\overline{A}}
\Mdef{\Bb}{\overline{B}}
\Mdef{\Cb}{\overline{C}}
\Mdef{\Db}{\overline{D}}
\Mdef{\Eb}{\overline{E}}
\Mdef{\Fb}{\overline{F}}
\Mdef{\Gb}{\overline{G}}
\Mdef{\Hb}{\overline{H}}
\Mdef{\Ib}{\overline{I}}
\Mdef{\Jb}{\overline{J}}
\Mdef{\Kb}{\overline{K}}
\Mdef{\Lb}{\overline{L}}
\Mdef{\Mb}{\overline{M}}
\Mdef{\Nb}{\overline{N}}
\Mdef{\Ob}{\overline{O}}
\Mdef{\Pb}{\overline{P}}
\Mdef{\Qb}{\overline{Q}}
\Mdef{\Rb}{\overline{R}}
\Mdef{\Sb}{\overline{S}}
\Mdef{\Tb}{\overline{T}}
\Mdef{\Ub}{\overline{U}}
\Mdef{\Vb}{\overline{V}}
\Mdef{\Wb}{\overline{W}}
\Mdef{\Xb}{\overline{X}}
\Mdef{\Yb}{\overline{Y}}
\Mdef{\Zb}{\overline{Z}}

\Mdef{\db}{\overline{d}}
\Mdef{\hb}{\overline{h}}
\Mdef{\qb}{\overline{q}}
\Mdef{\rb}{\overline{r}}
\Mdef{\tb}{\overline{t}}
\Mdef{\ub}{\overline{u}}
\Mdef{\vb}{\overline{v}}

\Mdef{\hc}{\hat{c}}
\Mdef{\he}{\hat{e}}
\Mdef{\hf}{\hat{f}}
\Mdef{\hA}{\hat{A}}
\Mdef{\hH}{\hat{H}}
\Mdef{\hJ}{\hat{J}}
\Mdef{\hM}{\hat{M}}
\Mdef{\hP}{\hat{P}}
\Mdef{\hQ}{\hat{Q}}

\Mdef{\thetab}{\overline{\theta}}
\Mdef{\phib}{\overline{\phi}}

\Mdef{\uA}{\underline{A}}
\Mdef{\uB}{\underline{B}}
\Mdef{\uC}{\underline{C}}
\Mdef{\uD}{\underline{D}}

\Mdef{\bolda}{\mathbf{a}}
\Mdef{\boldb}{\mathbf{b}}
\Mdef{\bfD}{\mathbf{D}}

\Mdef{\fm}{\frak{m}}
\Mdef{\fp}{\frak{p}}
\newcommand{\fX}{\mathfrak{X}}

\Mdef{\eps}{\epsilon}

\input{xypic}
\newcommand{\Gspec}{\mbox{$\mathbf{G}${\bf -spectra}}}
\newcommand{\Gspecc}{\mbox{$\mathbf{G}${\bf -spectra}}^c}

\newcommand{\GspecK}{\mbox{$\mathbf{G}${\bf -spectra}}\mathbf{\langle   K\rangle}}

\newcommand{\NspecK}{\mbox{$\mathbf{N}${\bf -spectra}}\mathbf{\langle K\rangle}}
\newcommand{\freeGspec}{\mbox{{\bf free-}$G${\bf -spectra}}}
\newcommand{\freeWspec}{\mbox{{\bf free-}$\mathbf{W}${\bf -spectra}}}

\newcommand{\loct}{\mathrm{Loc}_{\otimes}}
\newcommand{\thickt}{\mathrm{Thick}_{\otimes}}
\newcommand{\Ig}{\mathcal{I}_g}
\newcommand{\Iun}{\mathcal{I}_{un}}
\newcommand{\elr}[1]{E\langle #1 \rangle}
\newcommand{\sub}{\mathrm{Sub}}
\newcommand{\Lcl}{\Lambda_{cl}}
\newcommand{\Lct}{\Lambda_{ct}}

\newcommand{\torsHBGWmod}{\mbox{tors-$H^*(BG_e)[G_d]$-mod}}
\newcommand{\ctmax}{\mathrm{max}_{ct}}
\begin{document}
\title{The Balmer spectrum of rational
  equivariant cohomology theories}

\author{J.P.C.Greenlees}
\address{School of Mathematics and Statistics, Hicks Building, 
Sheffield S3 7RH. UK.}
\email{j.greenlees@sheffield.ac.uk}
\date{}

\begin{abstract}
The category of rational $G$-equivariant cohomology theories for a
compact Lie group $G$ is the homotopy category of rational $G$-spectra
and therefore tensor-triangulated. We show that its Balmer spectrum is
the set of conjugacy classes of closed subgroups of $G$, with the topology
corresponding to the topological poset  of \cite{ratmack}. This is used to
classify  the collections of subgroups arising as the geometric
isotropy of finite $G$-spectra. The ingredients for this
classification are (i) the algebraic
model of free spectra of the author and B.Shipley \cite{gfreeq2}, (ii) the Localization
Theorem of Borel-Hsiang-Quillen and (iii) tom Dieck's calculation of
the rational Burnside ring \cite{tD}. 
\end{abstract}

\thanks{I am grateful to P.Balmer and B.Sanders for
  conversations about this project, and especially to H.Krause, whose
  five-minute summary of the theory of support at the end of the
  2017 Copenhagen Workshop led to my enlightenment. Thanks also to 
G. Stevenson for the proof of the tensor version of the Thomason
Localization Theorem and to M.Kedziorek for comments on an earlier version.  }
\maketitle

\tableofcontents

\section{Introduction}
\subsection{Context}
This paper relates the general structure of the category of 
 $G$-equivariant cohomology theories for a compact Lie group $G$ to 
the structure of the Lie group $G$. To start with, the category is tensor-triangulated, since it is the
homotopy category of the monoidal model category of $G$-spectra. To see the broad
features of this category we restrict attention to cohomology theories
whose values are rational vector spaces:  a $G$-equivariant cohomology theory with
values in rational vector spaces  is represented by a $G$-spectrum
with rational homotopy groups. 

The crudest structural features of this category are reflected in the localizing
subcategories and the thick subcategories of compact objects. These in
turn are based on an understanding of the Balmer spectrum. We show
that the Balmer primes are in bijective correspondence with conjugacy
classes of  closed
subgroups of $G$, and we identify both   the containments amongst
them  and the topology. These correspond precisely to  the structures identified in
\cite{ratmack}.

It turns out that to prove these results we only need an
understanding of free spectra for various groups, as described in
\cite{gfreeq2}. Accordingly the arguments presented here are clearer
and more elementary than those presented in \cite{spctq} for tori, and
simultaneously give results for all compact Lie groups $G$. The
present paper renders \cite{spctq} obsolete.

In retrospect this result gives an intrinsic justification for the
form of the full algebraic models of \cite{s1q, tnq1, tnq2, tnqcore,
  s1q, so3q}, and suggest a revisionist approach to the general
project of giving an algebraic model for rational $G$-spectra. 

\subsection{Tensor triangulated categories}
We recall some standard terminology from the study of 
tensor-triangulated categories (tt-categories) and  the basic
definitions from \cite{Balmer1}. 

If $\C$ is a tensor triangulated category, an object $T$ is called  {\em small}, {\em finite} or {\em 
compact} if  for any set of objects $Y_i$, the natural map 
$$\bigoplus_i [T,Y_i]\stackrel{\cong}\lra [T, \bigvee_i Y_i]$$
is an isomorphism (where $[A,B]$ denotes the $\cC$-morphisms from $A$
to $B$). We write $\C^c$ for the tensor triangulated subcategory of
compact objects.  The word `compact' has become unavoidable by virtue of the 
superscript $c$ that it engendered, but we will often use the word 
`finite' since it is suggestive of a finite $CW$-complex. 

We say that  a full subcategory $A$ of $\C$
is {\em thick} if it is closed under completing triangles and taking retracts. 
  It is
{\em localizing } if it is closed under completing triangles and
taking arbitrary
coproducts (it is then automatically closed under retracts as
well).  We say that $A$ is an {\em ideal} if it is closed under
triangles and tensoring with an arbitrary element. 

For a general subcategory $B$ we write $\thick(B)$ for the
thick subcategory generated by $B$  and $\thickt(B)$ for the
thick tensor ideal generated by $B$. The latter depends on the ambient
category, and we will only write $\thickt(B)$ in the category $\C^c$ of
compact objects (so $B$ is compact, and only tensor products with
compact objects are permitted).  We write $\loc(B)$ for the 
localizing subcategory generated by $B$, and
$\loct(B)$ for the  localizing tensor ideal generated by $B$; because
an infinite coproduct of compact objects will usually not be compact, 
 this only makes sense for the full category $\C$ and tensor products with 
arbitarary objects of $\C$ are permitted. 
 
We will generally be interested in thick and localizing tensor
ideals, because without closure under tensor products  the structure
is hard to understand. We will give an example to illustrate
this in a special case in Section \ref{sec:semifree}.

\begin{defn}
A {\em prime ideal} in a tensor triangulated category is a thick
proper tensor ideal $\wp$ with the
property that $a\tensor b\in \wp$ implies that $a$ or $b$ is in $\wp$. 

The {\em Balmer spectrum} of a tensor-triangulated category $\C $ is
the set of prime tensor ideals of the triangulated category of compact objects: 
$$\spc (\C)=\{ \wp \subseteq \C^c \st \wp \mbox{  is prime } \}.$$ 

We may use this to define the {\em support} of a compact object:
$$\supp (X)=\{ \wp \in \spc(\C)\st X\not\in \wp\}. $$
This in turn lets us define the  Zariski topology on $\spc (\C)$ as 
generated by the closed sets $\supp (X)$ as  $X$ runs through compact objects of $\C$. 
\end{defn}

\begin{example}
The motivating example is that if $\C =D(R)$ is the derived category of a
commutative Noetherian ring $R$ then there is a natural homeomorphism
$$\spec (R)\stackrel{\cong}\lra \spc (D(R))$$
where the classical algebraic prime $\wp_a$ corresponds to the Balmer
prime $\wp_b=\{ M \st M_{\wp_a}\simeq 0\}$. To avoid disorientation it
is essential to emphasize that   this is order-reversing, so that maximal algebraic primes
correspond to minimal Balmer primes; either way these are the closed
points. 
\end{example}

\subsection{Transformation groups}

Our classification is in effect  in terms of traditional
invariants of transformation groups, namely fixed points and Borel
cohomology.  

If $A$ is a based $G$-space and $K$ is a subgroup of $G$, the fixed point
space $A^K$ admits an action of the Weyl group $W_G(K)=N_G(K)/K$ of $K$.  
We will make constant use of the extension of this functor to
$G$-spectra, which is the $K$-geometric fixed point functor
$\Phi^K$. It is an extension in the sense that 
$\Phi^K(\Sigma^{\infty}A)\simeq \Sigma^{\infty}(A^K)$. We will
generally  omit notation for the suspension spectrum, and accordingly 
write $\Phi^KA$ for the fixed point space as well as the associated
suspension spectrum. 

The functor $\Phi^K$ has other familiar properties in that it is a
tensor triangulated functor: it preserves triangles and 
$\Phi^K(X\sm Y)\simeq \Phi^KX\sm \Phi^KY$.

This is used to define the {\em geometric isotropy}\footnote{
In \cite{assiet} and the author's subsequent work this was called {\em
 stable isotropy} to distinguish it from the usual unstable notion. The corresponding notion for
categorical fixed points does not seem to be useful, so this caused no
confusion. 

The name of `geometric isotropy' from \cite{HHR} seems to have
acquired currency, and the symmetry between `stable' and `unstable'
does not seem sufficient to overturn this advantage. 
}
of a $G$-spectrum $X$:
$$\Ig (X)=\{ K\st \Phi^KX\not \simeq_1 0\} $$
is the collection of closed subgroups $K$ for which the geometric fixed 
points $\Phi^KX$ are non-equivariantly essential.

The geometric isotropy is an excellent way to organize our
understanding of $G$-spectra. In particular,  a (homotopically) {\em free}
$G$-spectrum is one which is either contractible or has geometric isotropy 
$\{ 1\}$.

We are especially interested in the category of rational equivariant
cohomology theories. Each rational equivariant cohomology theory 
$E_G^*(\cdot)$ is represented in the sense that there is a rational
$G$-spectrum $E$ so that for any based
$G$-space $X$, 
$$E_G^*(X)=[X,E]_G^*. $$
 More precisely, there is a stable
symmetric monoidal model category of rational $G$-spectra.  Its homotopy
category $\Gspec$ is tensor triangulated, and equivalent to the
category of rational equivariant cohomology theories and  stable natural
transformations. We will work throughout at the level of tensor
triangulated categories. 

\subsection{The Balmer spectrum}

There are some obvious primes in the category of $G$-spectra: for any
closed subgroup $K$ of $G$, we take
$$\wp_K=\{ X\st \Phi^KX\simeq_1 0 \}. $$
To see this is prime we note that $0$ is a prime in homotopy category
of finite rational spectra
(since that is equivalent to the derived category of $\Q$-modules) and 
$$\wp_K=(\Phi^K)^*((0)) \mbox{ where } \Phi^K: \Gspec \lra
\mbox{\bf spectra}. $$
This gives a prime for each closed subgroups $K$ of $G$, and 
conjugate subgroups give the same primes. In fact this gives all
primes, and we may describe the containments between them.

 We say that $L$ is
{\em cotoral} in $K$ if $L$ is a normal subgroup of $K$ and $K/L$ is a
torus. 

\begin{thm} 
\label{thm:Balmerposet}
The Balmer spectrum of prime thick tensor ideals  in the category of
finite rational $G$-spectra  is in bijective correspondence
to the closed subgroups of $G$. Containment corresponds to cotoral
inclusion:
$$\wp_K\subseteq \wp_H \mbox{ if and only if } K \mbox{ is conjugate
  to a subgroup cotoral in } H.$$
The Zariski topology of $\spc (\Gspec)$ is the Zariski topology on the $f$-topology from \cite{ratmack}. 
\end{thm}

The $f$-topology and the Zariski topology it generates will be
explained in Section \ref{sec:top}, where the proof will also be completed. For now we just remark
that  the fact that $\wp_K\subseteq \wp_H$ if $K$ is cotoral in $H$ comes 
from the classical Borel-Hsiang-Quillen Localization Theorem. The
reverse implication comes from tom Dieck's calculation of the rational 
Burnside ring. 
It  is a remarkable vindication of the Balmer spectrum that  it  captures the space of subgroups
and  cotoral inclusions,  and even  the $f$-topology of \cite{ratmack}. This can be put 
down to the fact that both the Balmer spectrum and  the analysis of
rational $G$-spectra are  principally based on the Localization theorem and the
calculation of the rational Burnside ring. 

\begin{remark}
In the light of Theorem \ref{thm:Balmerposet}, for any finite spectrum $X$, the support in the sense of Balmer for
this set of primes coincides with the geometric isotropy:
$$\supp (X)=\{ H \st X\not \in \wp_H \}=\{ H \st \Phi^HX \not  \simeq_1
0\}=\Ig (X). $$
\end{remark}

\subsection{Classification of thick tensor ideals}

Continuing with finite spectra,  we classify the
finitely generated thick tensor ideals in $\Gspectra$.

\begin{thm} 
(i) If $X$ is a finite rational $G$-spectrum then 
then $\Ig (X)$ is closed under passage to cotoral
subgroups and its space of cotorally maximal elements is open and compact in
the $f$-topology. 

(ii) Any set of subgroups which is closed under cotoral
specialization and whose set of cotorally maximal elements is open and
compact in the $f$-topology
occurs as $\Ig (X)$ for some finite rational $G$-spectrum $X$.

(iii) If $X$ and $Y$ are finite rational $G$-spectra with
$\Ig (Y)\subseteq \Ig (Y)$ then $Y$ is in the thick tensor ideal
generated by $X$. 
\end{thm}

One of the things coming out of this is the importance of `basic cells'. First, we note that 
rationally the {\em natural} cells $G/K_+$ are often decomposable, and
are finite wedges of certain {\em basic} cells
(\cite{tnq1}, see Subsection \ref{subsec:basic}).  In one sense these
are embodiments of the topology on the space of conjugacy classes, and
they can be used as the basis for a theory of cell complexes and weak
equivalences. 

\subsection{Classification of localizing tensor ideals}
In fact the strategy of proof is to begin by considering infinite
spectra and then deduce how finite spectra behave,  using Thomason's Localization 
Theorem (recorded here as \ref{thm:TLT}).

\begin{thm}
\label{thm:loct}
The localizing tensor ideals of $\Gspec$ are in bijective
correspondence with arbitrary collections of conjugacy classes of closed subgroups of $G$. The localizing
tensor ideal corresponding to a collection $\mcH$ of subgroups closed
under conjugacy is 
$$\Gspec \langle \mcH\rangle =\{ X \st \Ig (X)\subseteq \mcH\}. $$
\end{thm}

For each $K$ we may consider the localizing subcategory
$$\GspecK=\{ X\st \Ig(X)\subseteq (K)_G \}$$
of $G$-spectra `at $K$'. Theorem \ref{thm:loct} shows that $\GspecK$
is minimal in the 
sense that it is generated by any non-trivial object. This statement follows from an understanding of free
$W_G(K)$-spectra via \cite{gfreeq2}, and  is the key ingredient in proving 
Theorem \ref{thm:loct}. Indeed, it is also the main step in identifying the Balmer
spectrum as a set, and in the classification of thick tensor ideals of
finite spectra.

\subsection{Notation}
We will tend to let subgroups follow the alphabet, so that 
$$G\supseteq H\supseteq K \supseteq L.$$
and the trivial group is denoted $1$. We write $(H)_G$ for the
$G$-conjugacy class of $H$ and we write $L\subseteq_GK$ if $L$ is
$G$-conjugate to a subgroup of $K$.

We write $\sub (G)$ for the set of closed subroups of $G$, and
$\sub(G)/G$ for the set of conjugacy classes. We write $\cF G
\subseteq \sub (G)$ for the set of subgroups of finite index in their
normalizer, and $\Phi G =\cF G/G$ for the corresponding set of
conjugacy classes. 

Given a partially ordered set ({\em poset}) $X$ and a subset
$A\subseteq X$, we write $\Lambda_{\leq}(A)=\{ b\in X\st b\leq a
\mbox{ for some } a\in A\}$
for the downward closure of $A$. 

We will consider two partial
orderings on the set of  closed subgroups of a compact Lie group $G$.
We indicate containment of subgroups by $K\subseteq H$,  and refer to
this as the {\em classical} ordering. Accordingly $\Lcl (K)$ consists
of all closed subgroups of $K$. More often we consider conjugacy
classes of subgroups and the ordering induced by subconjugacy. 
The more signifcant ordering in this
paper  is
that of  {\em cotoral inclusion}. We write $K\leq H$
if $K$ is normal in $H$ and $H/K$ is a torus. The importance of this
ordering arises from the Localization Theorem. The set 
$\Lct (K)$ consists of all closed  subgroups cotoral in $K$.
More often we consider conjugacy
classes of subgroups and the ordering induced by being cotoral up to
 conjugacy.

 All homology and cohomology will be reduced and have rational
 coefficients. All spectra will be rational. 

A $G$-equivalence of $G$-spectra will be denoted $X\simeq Y$. A non-equivariant
equivalence of $G$-spectra will be denoted $X\simeq_1 Y$ for emphasis.

\part{Infinite spectra}

\section{Isotropy separation}
The purpose of this section is to recall some well known facts about
universal spaces in equivariant topology: the spaces with single
subgroup isotropy are basic building blocks. This corresponds to the
way that to each prime in commutative algebra there is an
indecomposable injective, and that all modules can be constructed from
them. 

\begin{defn}
Writing $\{ \subseteq_GK\}$ for the family of subgroups subconjugate
to $K$ and similarly for proper subgroups, for any subgroup $K$ of $G$ we may define the $G$-space
$\elr{(K)_G}$ by the cofibre sequence
$$E\{ \subset_G K\}_+\lra E\{ \subseteq_G K\}_+\lra \elr{(K)_G}$$
\end{defn}

It is clear from the definition that the geometric isotropy of
$\elr{(K)_G}$ is precisely the conjugacy class $(K)_G$ and that
$\Phi^K\elr{(K)_G}\simeq_1S^0$. Where the context makes clear the
ambient group we will write $\elr{K}=\elr{(K)_G}$.

\begin{lemma}
\label{lem:elrgen}
The $G$-spectra $\elr{K}$ as $K$ runs through closed subgroups of $G$
generate the category of $G$-spectra as a localizing tensor ideal. 
\end{lemma}

\begin{proof}
We consider the collection $\cC \cF$ of families $\cF$ so that $E\cF_+$ lies in
$\loc (\elr{K}\st K \subseteq G)$, ordered by inclusion. We show that
$\cF =All$ is in $\cC \cF$,
so that $S^0=EAll_+$ can be built from $\elr{K}$, and hence the other
cells can be obtained by using smash products. 

Any increasing chain $\{ \cF_{\alpha} \}_{\alpha}$ has an upper
bound. Indeed, a functorial construction of the universal space
(such as the bar construction) shows there is a functor from from families to $G$-spaces, and hence
the increasing chain gives a strict diagram,  and by considering fixed
points we see
$$E\left( \bigcup_{\alpha}\cF_{\alpha}\right)_+=\hocolim_{\alpha} [
(E\cF_{\alpha})_+]. $$
Since there are countably many conjugacy classes, we may assume the
chain is sequential and hence the homotopy colimit is in the
localizing subcategory of its terms. 

Finally we suppose that $\cF$ is maximal in $\cC \cF$.  If it is not
$All$, since descending chains of subgroups are finite,  we may
find a subgroup $K$ which not in $\cF$ but with all subgroups in
$\cF$. Then we have a cofibre sequence
$$E\cF_+ \lra E\cF\cup\{(K)_G\}_+\lra \elr{(K)_G}, $$
which contradicts maximality. 
\end{proof}

\section{Free $G$-spectra}
We have defined the category of free $G$-spectra to be the localizing
subcategory of $\Gspec$ consisting of the spectra with geometric
isotropy contained in $\{ 1\}$. This coincides with the localizing
subcategory generated by $G_+$. It is also equivalent to the homotopy
category of several well known model categories.

It was proved in \cite{gfreeq2} that a model category for free
$G$-spectra   is Quillen equivalent to an algebraic model. Writing
$G_e$ for the identity component of $G$ and $G_d=G/G_e=\pi_0(G)$ for
the discrete quotient, we see that $G_d$ acts on $G_e$ and hence we
may form the skew group ring $H^*(BG_e)[G_d]$. This gives rise to a
Quillen equivalence between free $G$-spectra and an algebraic model. 
This in turn induces an equivalence 
$$\freeGspec\simeq D(\torsHBGWmod). $$
of triangulated categories. 

Note that in concrete terms it means we have an efficient method of
calculation:  for any free $G$-spectra $X$ and $Y$, there is an Adams spectral sequence
 $$E_2^{*,*}=\Ext_{H^*(BG_e)}^{*,*}(H_*^{G_e}(X),
 H_*^{G_e}(X))^{G_d}\Rightarrow [X,Y]^G_*.$$
However for our purposes we only need a structural consequence.

\begin{thm}
The category 
$$D(\torsHBGWmod). $$
is generated as a localizing tensor ideal by any non-trivial elment. 
\end{thm}

\begin{proof}
The statement that the derived category of DG-modules over a graded
polynomial ring is generated by any non-trivial element is well known
as part of the classification of localizing subcategories of modules over
the polynomial ring. 

Suppose then that $M$ is a non-trivial element of
$D(\torsHBGWmod)$. We form $M[G_d]:=\Q [G_d]\tensor M$ and its fixed
point set $M':=M[G_d]^{G_d}$. From the algebraic result, this
generates all  torsion modules over $H^*(BG_e)$. Hence $M[G_d]\cong
M'[G_d]$ constructs all torsion modules of the form $N'[G_d]$ for 
a DG-torsion $H^*(BG_e)$-module $N'$. However
any torsion $H^*(BG_e)[G_d]$-module $N$ is a retract of 
$N'[G_d]$ where $N'=N[G_d]^{G_d}$.  Hence $M$ generates the whole
category as a localizing tensor ideal.
\end{proof}

\begin{cor}
\label{cor:freemin}
For any compact Lie group $G$, the category of free rational $G$-spectra
is generated as a localizing tensor ideal by any non-trivial elment. 
\end{cor}

\begin{proof}
Since the equivalence of \cite{gfreeq2} was not shown to be monoidal, a
few words of  proof are required. 

If $X$ is a non-trivial free $G$-spectrum then so is $(G/G_e)_+\sm
X$. This has homotopy which is free over $\Q [G_d]$. Applying the
theorem we note that the corresponding
DG-torsion-$H^*(BG_e)[G_d]$-module generates all such modules as a
triangulated category. It follows that any free $G$-spectrum of the
form $(G/G_e)_+\sm Y$ is in the localizing tensor ideal generated by
$X$. Since $Y$ is a retract of $(G/G_e)_+\sm Y$, this completes the
proof. 
\end{proof}

\section{Localizing tensor ideals}
Recall that we write $\Gspec$ for the homotopy category of rational $G$-spectra
and $\GspecK$ for the category of rational $G$-spectra which are either contractible or 
have  geometric isotropy conjugate to $K$. We will show that $\GspecK$ is
generated as a localizing tensor ideal by any non-trivial element.
 
\subsection{Spectra over a normal subgroup}
\label{subsec:GspectraoverK}
It is well known that if $\Gamma$ is a compact Lie group with normal
subgroup $\Delta $ then  geometric fixed points induce an equivalence 
$$\mbox{$\Gamma$-spectra $\langle \supseteq \Delta\rangle $}\simeq \mbox{$\Gamma/\Delta$-spectra} $$
of tensor triangulated categories. The first category consists of
$\Gamma$-spectra with geometric isotropy consisting of subgroups
containing $N$ (sometimes called $\Gamma$-spectra `over $\Delta$'). 

Given a closed subgroup $K$ of $G$,  we apply this  when
$\Gamma =N_G(K)$, $\Delta=K$. In particular we have an isomorphism 
$$\Phi^K: [X,Y]^{N_G(K)}\stackrel{\cong}\lra
[\Phi^KX,\Phi^KY]^{W_G(K)} $$
whenever the geometric isotropy of $Y$ lies in $\langle \supseteq
K\rangle $.
This applies integrally. 

\subsection{Restriction}
Suppose $G\supseteq H\supseteq K$. 
If we have a $G$-spectrum $X$, we may restrict the $G$-conjugacy class
$(K)_G$  to a collection of subgroups of $H$. This may
well be bigger than $(K)_H$ and the subgroups of the  $G$-conjugacy
class $(K)_G$ which lie inside $H$ may break into several $H$-conjugacy 
classes. 

\begin{example}
 If $G=SO(3)$ then all elements  of order 2 are conjugate,
 and we may take $K$ to be generated by a half turn around the
 $z$-axis. Now take $H=O(2)$  to be the normalizer of $K$. The 
involutions in $O(2)$ fall into two conjugacy classes: the rotations 
of order 2  (actually a singleton) and the reflections (forming a
space homeomorphic to a 
circle). Accordingly, when we restrict a spectrum $X$ with geometric 
isotropy  $H$ to an $N$-spectrum, the geometric isotropy will no
longer be a single conjugacy class. 
\end{example}

The example is typical. 

\begin{lemma}
There is a finite decomposition
$$(K)_G\cap \sub(H)=\coprod_{i=0}^n (K_i)_H,  $$
where $K_i=K^{\gamma_i}$ for $\gamma_i \in G$ and $\gamma_0=e$. 
\end{lemma}
\begin{proof}
It is clear that $(K)_G$ breaks into a disjoint union of $H$-conjugacy
classes, each of which is a closed subset of the space of subgroups
with the Hausdorff metric topology. By the Montgomery-Zippin Theorem (`close-means-subconjugate')
they are also open, so there are finitely many. 
\end{proof}

\begin{remark}
This corresponds to the fact that 
$$\Phi^K(G_+\sm_HY)\simeq \bigvee_{i=0}^n \gamma_i \Phi^{K_i}Y. $$
\end{remark}

Now we specialize to $H=N=N_G(K)$. 

\begin{lemma}
\label{lem:eKU}
There is an open and closed subset $U\subseteq  \Phi N$ so that 
$K\in \Lct (U)$ but $K_i\not \in \Lct (U)$ for $i\neq 0$. 
\end{lemma}

\begin{proof}
It suffices to prove that if $K'\subseteq N$ is a $G$-conjugate of $K$
distinct from $K$, then 
there is no subgroup of $N$ in which both $K$ and $K'$ are
cotoral. It suffices to show that if the image $K'/K\cap K'$ of $K'$ in
$W=N/K$ is a torus then it is trivial. 

For this, we note that since the maximal tori of $K$ and $K'$ are subtori of $N$
which are $G$-conjugate, they are also $N$-conjugate \cite[13.4]{AGtoral},
so we may suppose that $K\cap K'$ is their common maximal
torus. Accordingly if $K'/K\cap K'$ is a torus, it must be trivial. 
\end{proof}

Now choose an open and closed subspace $U$ as in Lemma \ref{lem:eKU} and let 
$e^G_K\in A(N)$ denote the corresponding idempotent. 

\begin{lemma}
\label{lem:restonormalizer}
For $G$-spectra $X$ and $Y$ with geometric isotropy $K$, restriction to $N_G(K)$ 
induces an  isomorphism
$$[X,Y]^G\stackrel{\cong}\lra [e^G_KX, e^G_KY]^N $$
\end{lemma}
\begin{proof}
The  restriction 
$$[X,Y]^G\stackrel{\cong}\lra [X,Y]^{N_G(K)}. $$
is induced by the projection map $G_+\sm_Ne^G_KS^0 \lra G/G_+$, 
which is an equivalence in geometric  $K$-fixed points by construction. 
\end{proof}

\subsection{From $G$-spectra at $K$}
Assembling the above information we may 
understand maps of $G$-spectra with geometric isotropy
$K$ in terms of their geometric $K$-fixed points. For brevity we 
write $N=N_G(K)$ and $W=W_G(K)$, and we consider the two maps
$$[X,Y]^G\stackrel{\res^G_N}\lra [X,Y]^N\stackrel{\Phi^K}\lra 
[X,Y]^W. $$

\begin{lemma}
\label{lem:resfix}
For $G$-spectra $X$ and $Y$ with geometric isotropy $K$, restriction to $N_G(K)$ and passage to geometric $K$-fixed points 
induce  isomorphisms 
$$[X,Y]^G\stackrel{\cong}\lra [e^G_KX, e^G_KY]^N \stackrel{\cong}\lra 
 [\Phi^KX, \Phi^KY]^{W_G(K)}$$
\end{lemma}

\begin{proof}
The first isomorphism is Lemma \ref{lem:restonormalizer} 
 and the second is the fact from Subsection \ref{subsec:GspectraoverK}. 
\end{proof}

\begin{cor}
\label{cor:min}
The category $\GspecK$ is a minimal localizing subcategory of the
category of $G$-spectra in the sense that it is generated by any
non-trivial element. 
\end{cor}

\begin{proof}
By Theorem \ref{cor:freemin} the category of $\NspecK\simeq \freeWspec$ is a
minimal localizing subcategory of $N$-spectra. 

Now if $X$ is a non-trivial $G$-spectrum with geometric isotropy $K$
then $e^G_KX$  is non-trivial as an $N$-spectrum and therefore
generates all $\NspecK$. We may coinduce this construction to see that 
$F_N(G_+, e^G_KX)$ builds any spectrum of the form $F_N(G_+,
Y)$ for $Y$ in $\NspecK$. 

If $\cF$ is any family, the collection of $\cF$-spectra is generated
as a localizing category by the cells $G/L_+$ with $L$ in $\cF$. It is
also generated as a localizing category by the duals $DG/L_+$ for
$L\in \cF$.  Accordingly, if $\cF$ consists of subconjugates of
$N$, the $\cF$-spectra are built by objects $F_N(G_+,Z)$. Now if 
$Z=Z'\sm \elr{K}_N$ we note 
$$F_N(G_+, Z)\simeq \elr{(K)_G}\sm  F_N(G_+, Z'). $$
Since any object of $\GspecK$  is of the form $\elr{(K)_G}\sm Y$ with 
$Y$ built using subconjugates of $K$, so we conclude that $\GspecK$ is generated by the objects $F_N(G_+,Y)$
for  $Y$ in $\NspecK$.
 \end{proof}

\begin{thm}
\label{thm:loctideals}
The localizing tensor ideals of $G$-spectra are precisely the unions
of the single geometric isotropy category spectra $\GspecK$. 

For any $G$-spectrum $X$,  
$$\loct (X)= \sum_{K\in \Ig (X)}\GspecK$$
\end{thm}
 
\begin{proof}
The $G$-spectrum $\elr{K}$ has geometric isotropy precisely $(K)_G$. Thus
if $K\in \Ig (X)$ we find an object $\elr{K}\sm X \in \loct(X)$ with
geometric isotropy precisely $K$. Hence by the Minimality Theorem
\ref{cor:min}, $\loct(X)$ contains $\GspecK$. 

It follows from Lemma \ref{lem:elrgen} that $X$ lies in the localizing tensor ideal
generated
by $\elr{K} \sm X$ for all $K$. 
\end{proof}

\part{Finite spectra}

\section{The Localization Theorem}
\label{sec:locthm}

The main ingredient in understanding containment of primes is the Localization Theorem.

We revisit some basic facts from transformation groups in our
language. The basic tool is Borel
cohomology, $H^*_G(X):=H^*(EG\times_G X, EG\times_G pt)=H^*(EG_+\sm_GX)$.

The idea is that for finite spectra, geometric isotropy is
determined by Borel cohomology. It then follows from the Localization
Theorem that the geometric isotropy is closed under passage to cotoral
subgroups.

\begin{lemma}
\label{lem:fpBorel}
If $K$ is connected and $X$ is a  finite $K$-CW-complex, then $X$ is non-equivariantly contractible if and
only if $H^*_K(X)=0$. 
\end{lemma}

\begin{proof}
First, if $X$ is simply connected,  the Hurewicz theorem shows $X\simeq *$ if
and only if $H_*(X)=0$. This is equivalent to $H^*(X)=0$. 

Next, we have a fibration $X\lra EK\times_K X \lra BK$ so the Serre spectral
sequence shows that if $H^*(X)=0$ then also $H_K^*(X)=0$. Conversely,
since $K$ is connected the Eilenberg-Moore theorem gives
$$C^*(X)\simeq C^*(EK\times_K X)\tensor_{C^*(BK)} \Q,   $$
(where the tensor product is derived, and the cochains are unreduced). This shows that $H^*_K(X)=0$ implies $H^*(X)=0$. 
\end{proof}

It follows that $\Ig (X)$ can be detected from Borel cohomology of
fixed points. 

\begin{cor}
\label{cor:fpBorel}
If $X$ is finite, $K\in \Ig (X)$ if and only if
$H^*_{W_G^e(K)}(\Phi^KX)\neq 0$.\qqed
\end{cor}

For us the fundamental fact is the following consequence of the Localization Theorem. 

\begin{prop}
\label{prop:Igctclosed}
If $X$ is finite then 
$\Ig (X)$ is closed under passage to
cotoral subgroups. 
\end{prop}

\begin{proof}
The Localization Theorem states that if $T$ is a torus and $Y$ is a
finite $T$-CW-complex then 
$$H^*_T(Y)\lra H^*_T(\Phi^T Y )=H^*(BK)\tensor H^*(\Phi^T Y)$$
becomes an isomorphism when the multiplicatively closed set $\cE_T =\{
e(W)\st W^T=0\} $ of Euler classes $e(W)\in H^{|W|}(BT)$ is
inverted. The proof uses the fact that the $T$-space
$$S^{\infty V(T)}=\bigcup_{W^T=0}S^W, $$
has $H$-fixed points $S^0$ if $H=T$ and is contractible
otherwise. Accordingly, we have a $T$-equivalence
$$Y\sm S^{\infty V(T)}\simeq \Phi^TY \sm S^{\infty V(T)}. $$ 

 It follows that if
$H^*(\Phi^TY)\neq 0$ then also $H^*_T(Y)\neq 0$. 

Now suppose $H\in \Ig(X)$. Let $Y=\Phi^KX$ and let $T=H/K$. The
hypothesis states $\Phi^TY=\Phi^{H/K}\Phi^KX\not \simeq *$ so that 
so that $H^*_T(\Phi^TY) \neq 0$. By  the localization theorem it
follows that $H^*_T(Y)\neq 0$, and  from Lemma \ref{cor:fpBorel}
$\Phi^KX=Y\neq 0$. Hence also $K \in \Ig (X)$. 
\end{proof}

The Localization Theorem explains the importance of the cotoral
ordering. The first half is as follows. 
 
\begin{cor}
If $K\leq H$ then $\wp_K\subseteq \wp_H$. \qqed
\end{cor}

We will see in Lemma \ref{lem:primeordercotoralorder} below that the
reverse implication also holds.

\section{Burnside rings and basic cells}
\label{sec:basiccells}
A distinctive feature of working rationally is that there are usually  many
idempotents in the rational Burnside ring. We follow
through the implications of this for cell structures.

Integrally, the relevant cells are homogeneous spaces
$G/K_+$, and the relevant ordering of subgroups is classical containment. 
Rationally, the splitting of Burnside rings means that the
relevant cells are certain basic cells $\sigma_{K, U}$ (a retract of
$G/K_+$ introduced below) and the
relevant ordering of subgroups is cotoral inclusion. 

We begin by running through one approach to {\em classical} equivariant cell
complexes, and then introduce basic cells and  follow the same pattern to give the rational
analysis in terms of {\em basic} cells.

\subsection{Unstable classical recollections}
Classically, we are used to the idea that based $G$-spaces $P$ are formed from
cells $G/K_+$. The classical unstable  isotropy is defined by 
$$\Iun' (P)=\{ K \st \Phi^K P \neq * \}; $$
 it is not homotopy invariant, but it does have the obvious
property that it is closed under passage to subgroups. 

It is natural to move to a homotopy invariant notion 
$$\Iun (P)=\{ K \st  \Phi^K P \not \simeq  *\} . $$ 
This notion fits well with the cells we use, since
$$\Iun (G/K_+)=\{ L \st L\subseteq_G K\} =\Lcl((K)_G). $$
Note that we are linking  notions of cell and isotropy with
a partial order on subgroups.

The homotopy invariant version of unstable isotropy may not be closed under passage to subgroups, so that we only
know that $X$ is equivalent to a complex constructed from cells
$G/K_+$ with $K\in \Lcl\Iun (T)$, where $\Lcl$
indicates that we take the closure under the classical order (i.e.,
under containment). This can be proved by killing homotopy groups, or
by the method described in the next subsection for the stable
situation. 

\subsection{Stable classical recollections}
Moving to the stable world, for a $G$-spectrum $X$ we have the
 {\em geometric}  (or {\em stable}) isotropy
$$\Ig (X)=\{ K \st  \Phi^K X \not \simeq_1  0\} , $$
homotopy invariant by definition. Evidently since geometric fixed
points extend ordinary fixed points on spaces,  
$$\Iun (P) \supseteq \Ig (\Sigma^{\infty}P), $$
and if $P$ can be constructed from cells in $A$ then $\Sigma^{\infty}P
$ can be constructed from stable cells in $A$.

The attraction of geometric isotropy and geometric fixed points arises from 
the fact that their properties are familiar from the category of based spaces. 
Perhaps the most important instance is that of the Geometric Fixed Point Whitehead 
Theorem, and we state it here because it is fundamental to our approach. The result is
well known to all users of geometric fixed points, and is usually deduced
using isotropy separation to see that that an equivalence in all geometric fixed 
points is an equivalence in all categorical fixed points. 

\begin{lemma} {\em (Geometric Fixed Point Whithead Theorem)}
A map $f: X\lra Y$ of $G$-spectra is an equivalence if $\Phi^Kf: \Phi^KX\lra \Phi^KY$ 
is a non-equivariant equivalence for all closed subgroups $K$.\qqed
\end{lemma}

Next, we note that
$$\Ig (G/K_+) =\Lcl((K)_G).  $$
Any $G$-spectrum $X$ can be constructed from stable cells $G/K_+$ with
$K \in \Lcl \Ig (X)$. One way to prove this is to construct a filtration
analogous to the (thickened) fixed point filtration of a space.
Simplifying this, if $\cF =\Lcl\Ig (X)$ then $X\sm
\tilde{E}\cF$ has trivial   geometric fixed points (and is thus contractible by the Geometric Fixed Point
Whitehead Theorem). Now $X \sm E\cF_+$ may be constructed from cells
$G/K_+$ for $K \in \cF$. In effect we use the result for the special
case $E\cF_+$ together with the fact that $G/H_+ \sm G/K_+$ can be
constructed from cells $G/K'_+$ with $K'\subseteq K$.

We now see how we can to take advantage of the additional flexibility
of working rationally.

\subsection{The Burnside ring}
We recall tom Dieck's determination of the rational Burnside ring \cite{tD}. 
To any self-map $f:S^0\lra S^0$ we may associate the {\em mark}
$$m(f):\sub (G)\lra \Q$$
defined by 
$$m(f)(K)=\deg (\Phi^K f:S^0\lra S^0). $$
Obviously this is constant on conjugacy classes so we can view it as a
function on $\sub(G)/G$. More significantly, 
by the Localization Theorem if $L$ is cotoral in $K$ we have
$m(f)(L)=m(f)(K)$. Accordingly $m(f)$ is determined by its restriction
to the space of cotorally maximal conjugacy classes. Since every
infinite compact Lie group has a non-trivial torus,
$\ctmax(\sub(G)/G)$ is the space $\Phi G$ of conjugacy classes of subgroups of finite
index in their normalizers. We may topologize $\Phi G$ as the quotient
space of a space of closed subgroups with the Hausdorff metric
topology, and as such $m(f)$ is continuous. The mark homomorphism
$$m: [S^0,S^0]^G\stackrel{\cong}\lra C(\Phi G, \Q)$$
is an isomorphism of rings. 

 When $G$ is the product of a torus and a finite group $Q$
 this means  $A(G)\cong \prod_{(\overline{K})}\Q$, with the product
over conjugacy classes of subgroups $\overline{K}$ of
$\overline{G}=Q$. Accordingly we obtain one primitive
idempotent $e_{\overline{L}}$ for each conjugacy class of subgroups of $Q$.

In general $\Phi G$ is compact, Hausdorff and totally disconnected. Its topology
has a basis of open and closed sets $U$. Idempotents of $A(G)$
correspond to open and closed subsets: to each $U$ we write 
 $e_U$ for the idempotent with $U$ as support. 

\subsection{Basic cells}
\label{subsec:basic}
The classical cell $G/K_+$  is $S^0$ induced up from $K$, so if the
$K$-equivariant sphere decomposes, so does $G/K_+$. 

The building blocks are thus the {\em basic cells} 
$$\sigma_{K,U}:=G_+\sm_Ke^K_US^0, $$
where $U$ is an open and closed neighbourhood of $K$ in $\Phi K$.

\begin{remark}
If $K$ is isolated in $\Phi K$, we write $\sigma_K=\sigma_{K,
  \{K\}}$.  When $G$
is a torus,  all subgroups are of this type, and in general all 
subgroups of a  maximal torus are of this type. 
\end{remark}

We  develop cell structures  based on basic cells. The advantage  is that
the geometric isotropy of $\sigma_{K, U}$  is smaller than that of $G/K_+$.
The advantage is clearest when  $K$ is isolated in $\Phi K$, but in
general we can achieve similar results by letting $U$ range over
smaller and smaller neighbourhoods of $K$.

\begin{lemma}
\label{lem:gisigmaK} 
The geometric isotropy of $\sigma_{K, U}$ consists of all subgroups
$G$-conjugate to a subgroup cotoral in an element of $U$:
$$\Ig (\sigma_{K,U})=\Lct (U_G). $$
\end{lemma}

\begin{proof}
If $L\in U$ then $\Phi^{L} e_US^0\simeq_1 S^0$, and $Y$ is a
$K$-equivariant retract of $G_+\sm_KY$. Hence $L\in
\Ig(\sigma_{K,U})$. By Proposition \ref{prop:Igctclosed}, $\Lct (U)\subseteq \Ig
(\sigma_{K,U})$.

Conversely, we show $\Ig(\sigma_{K,U})\subseteq \Lct (U_G)$. 
Certainly $\Ig(\sigma_{K,U})\subseteq \Ig(G/K_+)\subseteq
\Lcl((K)_G)$. Next, if $L$ is a subgroup of $K$ not $G$-conjugate to
a subgroup cotoral in a subgroup in $U$ then there is an idempotent
$e$ orthogonal to $e_U$. Indeed, we may suppose $L$ is of finite index
in its normalizer (or else we replace it by the inverse image of the
maximal torus in $W_G(L)$ in $N_G(L)$). By hypothesis $(L)_G\cap
U=\emptyset$, and hence there is an open and closed subset $V$
containing it and still disjoint from $U$. Now
$\Phi^L(G_+\sm_KY)\simeq \Phi^L(G_+\sm_Ke_VY)$ and $e_Ve_US^0\simeq
0$.
\end{proof}

\section{Applications of basic cells}
With a little additional work on the topology of the space of
subgroups, basic cells are extremely valuable. 

\subsection{The $f$-topology}
We recall the $f$-topology on $\sub(G)$ from \cite{ratmack}. We write 
$d$ for  the Hausdorff metric on the space of closed subgroups, and if 
 $H$ is a closed subgroup of $G$ we consider 
$$O(H, \eps)=\{ K \subseteq H\st |W_H(K)|<\infty, d(H,K)<\eps\}$$
(i.e. only considering subgroups of finite index in their 
normalizers). A base of 
neighbourhoods of a subgroup $H$ in $\sub (G)$ consists of the sets 
$$O(H,A, \eps)=\bigcup_{a\in A}O(H,\eps)^a$$
where $A$ runs through neighbourhoods of the identity in $G$ and $\eps 
>0$. This topology induces the quotient topology on the space $\sub 
(G)/G$ of conjugacy classes, which we again call the $f$-topology. 

\begin{lemma}
The image $U_H^{\eps}$ of  $O (H,\eps)$ in $\Phi H$ is open and closed. The set 
of maximal elements in $\Ig (\sigma_{H, U_H^{\eps}})$ is the image of 
$O(H, \eps)$ in $\sub (G)/G$. \qqed
\end{lemma}

The following characterization of the topology may be helpful. 
\begin{lemma}
\label{lem:quottop}
The topology on $\sub(G)/G$ is the quotient topology for the maps 
$\Phi H\lra \sub(G)/G$ as $H$ runs through closed subgroups of $G$: 
a set is open in $\sub (G)/G$ if and only if its pullback to $\Phi 
H$ is open for all $H$. 
\end{lemma}

\begin{proof}
Lemma 8.6 (b) of \cite{ratmack} shows that the maps $\Phi H \lra \sub 
(G)$ are continuous. It remains to show that if $U$ is a subset of 
$\sub(G)/G$ which has the property $U\cap \Phi H  $ is open for all 
$H$ then $U$ is open. If $(K)_G \in U$ then the fact that $U\cap \Phi 
K$ is open shows it contains a neighbourhood $U_K^{\eps}$ of  $(K)_H$
for some $\eps >0$. The image of this in $\sub(G)/G$ is open in the 
$f$-topology. 
\end{proof}

The collection  of all closed subgroups of $G$ is a  poset under
cotoral inclusion. 

\begin{prop}\cite[8.7, 8.8]{ratmack}
Giving $\sub (G)$ the $f$-topology and topologizing the space 
$\sub_1(G)$ of cotoral inclusions as a subspace of $\sub(G)\times \sub (G)$, we obtain 
a topological category. The source and target maps are open maps. \qqed
\end{prop}

Note that for any set $A$ of subgroups,  $\Lct (A)=s(t^{-1}(A))$, so
that if $A$ is open so is $\Lct (A)$.

\subsection{Basic cells and primes}

The geometric isotropy of the basic cells determines the cotoral order.

\begin{cor}
\label{lem:primeordercotoralorder}
If $\wp_L\subseteq \wp_K$ then $L$ is conjugate to a subgroup cotoral
in $K$. 
\end{cor}

\begin{proof}
If $L$ is not cotoral we will show that there is a clopen
neighbourhood $U_K$ of $K$ in $\Phi K$ so that $\sigma_{K, U_K}\in
\wp_L\setminus \wp_K$. 

First note that   $s^{-1}(L)$ is closed. Its intersection with $\{ L\} \times \Phi K$ is compact, and so
$t(s^{-1}(L) \cap \{L \} \times \Phi K)$ is closed. Since it  does
not contain $K$ there is  a neighbourhood $U_K$
of $K$ in $\Phi K$ so that $(L)_G\cap \Lct (U_G)=\emptyset$ and hence
  $\sigma_{K,U}\in \wp_L$, and since $K \in U_K$ we see
  $\sigma_{K,U_K}\not \in \wp_K$. 
\end{proof}

These basic cells show that cotorally unrelated subsets can be 
separated. We will not use the following result elsewhere (which may 
excuse the forward reference in the proof). 

\begin{cor}
\label{cor:tzeroc}
If $K_1$ and $K_2$ are two subgroups cotorally unrelated (in the sense 
that $\Lct ((K_1)_G)\cap \Lct ((K_2)_G)=\emptyset$), then there 
are finite complexes $X_1$ and $X_2$ with $K_i \in \Ig(X_i)$ and $\Ig 
(X_1)\cap \Ig(X_2)=\emptyset$
\end{cor}

\begin{proof}
We will take $X_1=\sigma_{K_1, U_1}$ and $X=\sigma_{K_2, U_2}$. It 
remains to show  we can choose  neighbourhoods $U_1, U_2$ of  $K_1,
K_2$  that are also unrelated. 

Consider the set $V=\Ig (\sigma_{K_1, U_1}\sm \sigma_{K_2, U_2}))=\Lct
(K_1, U_1)\cap \Lct (K_2, U_2)$. We will show in Theorem
\ref{thm:ctmaxcomp}  below that $A=\ctmax{V}$ is
 open and compact. As above
$s^{-1}A$ intersects $\Phi K_i$ in a compact set, and so
$t(s^{-1}A\cap \Phi K_i)$ is closed and does not contain
$K_i$. Accordingly we can choose neighbourhoods $U_i$ of $K_i$ not
meeting it. 
\end{proof}

\subsection{Basic detection}
Basic cells play a comparable role to classical cells in that they
generate the category and detect equivalences. The smaller isotropy
means that we can make stronger statements. Note that the lemma 
holds however small the  neighbourhoods $U_K$ are. 

\begin{lemma}
\label{lem:basicbuildsclass}
If we choose an open and closed neighbourhood $U_K$ of $K$ in $\Phi K$
for each conjugacy class of closed subgroups of $G$ then 
the natural cell $G/K_+$ is built from basic cells $\sigma_{L, U_L}$ with 
$L\subseteq K$. 

Accordingly, the category of rational $G$-spectra is generated by the basic cells $\sigma_{K,U_K}$.
\end{lemma}

\begin{proof}
It suffices to show that the classical cells $G/K_+$ are built from
the basic cells. The proof is by induction on $K$ (i.e., we work with the
poset of all subgroups ordered by inclusion, and note that there are no
infinite decreasing chains). 

Since $G/1_+=\sigma_1$, the induction begins. Now suppose $K$ is
non-trivial  and that 
$G/L_+$ is built from basic cells for proper subgroups $L$ of $K$. Now
$G/K_+$ is a sum of $\sigma_{K,U_K}$ and the spectra $G_+\sm_K
e_{K,U_K'}S^0$ where $U_K'$ is the complement of $U_K$. 
Since $K\not \in \Ig (\sigma_{K, U_K'})$, we find $\sigma_{K,U_K'}$ is
built from cells $G/L_+$ where $L$ is a {\em proper} subgroup of $K$.
By induction it is thus built from a $\sigma_{L,U_L}$. 
\end{proof}

There is a useful criterion for vanishing of homotopy in terms of
geometric isotropy. 

\begin{lemma}
\label{lem:disjtsupp}
(i) If $\Lcl (K) \cap \Ig (X) =\emptyset$ then $[G/K_+,X]^G_*=0$. 

(ii) If $\Lct (K,U)\cap \Ig (X)=\emptyset$ then $[\sigma_{K,U},X]^G_*=0$
\end{lemma}

\begin{proof}
The first statement is immediate from the Geometric Fixed Point
Whitehead Theorem, since $\Ig (\res^G_KX)=\emptyset$.

For the second, we note 
$$[\sigma_{K,U} ,X]^G_*=[e_{K,U}S^0, X]^K_*=[e_{K,U}S^0, e_{K,U}X]^K_*. $$
By hypothesis $\Ig (e_{K,U}X)=\Lct (K, U)\cap \Ig (\res^G_K X)=\emptyset$,
so that $e_{K,U}X\simeq_K 0$ by the Geometric Fixed Point Whitehead Theorem. 
\end{proof}

A slightly refined version  of  the Whitehead Theorem holds in the
rational context: to establish an equivalence we need only check on
basic $K$-homotopy groups for $K$ in the cotoral closure of the
geometric isotropy rather than for the full classical closure. 

\begin{prop}
\label{prop:basicWHT}
Suppose that $\Ig (X), \Ig (Y)\subseteq \cK$ for a set $\cK$ of
subgroups,  and that for some open neighbourhoods $U_K$,  the map 
$f:X\lra Y$ induces an isomorphism of $[\sigma_{K, U_K}, \cdot ]^G_*$ for all
$K\in \cK$. Then $f$ is an equivalence.  
\end{prop}

\begin{proof}
Taking $Z$ to be the mapping cone of $f$, it suffices to show that if 
$\Ig (Z)\subseteq \cK$ and  $[\sigma_{K, U_K} , Z]^G_*=0$ for all $K\in \cK $ then $Z\simeq 0$.
We will show that in fact $\Ig (Z)=\emptyset$.  If not, there is a
minimal counterexample, $K\in \Ig(Z)$.  We then have 
$$0=[\sigma_{K,U_K}, Z]^G=[e_{U_K}S^0, Z]^K. $$
There is a $K$-equivariant cofibre sequence
$$E\cP_+\sm Z\lra Z \lra \etp \sm Z, $$
and it suffices to argue $[e_{U_K}S^0, E\cP_+ \sm Z]^K=0$ since then 
we have
$$0= [e_{U_K}S^0,  Z]^K=[e_{U_K}S^0, \etp \sm Z]^K=[S^0, \Phi^K Z], $$
where the last equality is because $K\in U_K$. 

We have 
$$ [e_{U_K}S^0, E\cP_+ \sm Z]^K=[e_{U_K}, e_{U_K}E\cP_+ \sm Z]^K$$
and $e_{U_K}E\cP_+\sm Z\simeq 0$ by minimality of $K$. 
\end{proof}

\subsection{Basic structures}
When we work rationally, classical containment of subgroups is
replaced by cotoral inclusion. Cells $G/K_+$ are now often
decomposable, and we have basic cells $\sigma_{K, U_K}$. If $K$ is a torus
then $\sigma_K=G/K_+$, but if $K$ is not a torus then there is a base
of neighbourhoods of $K$ for any one of  which we have a proper inclusion
$$\Ig (\sigma_{K, U_K})=\Lct (K, U_K)\subset \Lcl(K)=\Ig (G/K_+). $$

\begin{lemma}
\label{lem:constructfrombasicingi}
Any $G$-spectrum $X$ can be constructed from basic cells $\sigma_{K, U_K}$
with $K$ in $\Lct (\Ig (X))$. 
\end{lemma}

\begin{proof}
Take $\cK =\Lct \Ig (X)$.  We may construct a map $p:P\lra X$ so that $P$ is a wedge of suspensions
of basic cells $\sigma_{K, U_K}$ for $K\in \cK$, and so that $p_*$ is surjective on $[\sigma_{K,U_K}, \cdot
]^G_*$ for all $K \in \cK$. Iterating this, we form a diagram 
$$\diagram
X \ar@{=}[r]&X_0\rto &X_1\rto &X_2\rto &\cdots \\
&P_0\uto &P_1\uto &P_2\uto &
\enddiagram$$

We take $X_{\infty}=\hocolim_s X_s$, and note that 
since $\sigma_{K,U_K}$ is small for each $K$, it follows
 $[\sigma_{K, U_K}, X_{\infty}]^G_*=0$ for $K\in \cK$.
Since  $\Ig (X_{\infty})\subseteq \Lct \Ig (X)=\cK$, it follows from
Proposition \ref{prop:basicWHT} that $X_{\infty}\simeq 0$. Arguing with the
 dual tower, we see $X$ can be constructed from cells $\sigma_{K,U_K}$:
 indeed, we define $X^s$ by the cofibre sequence
$$X^s\lra X\lra X_s. $$
By definition  $X^0\simeq *$, and we have
$$\Sigma^{-1}P_s \lra X^s\lra X^{s+1}. $$
Again $X^{\infty}=\colim_s X^s$, and since $X_{\infty}\simeq 0$, we see
$X^{\infty}\simeq X$. 
\end{proof}

\begin{remark}
One can imagine other proofs. One is to construct an analogue of the
map  $E\cF_+ \lra S^0$ for a family $\cF$. This is a spectrum 
$\elr{\cK}$ with geometric isotropy $\cK$ and a map $\elr{\cK}\lra S^0$ which is an equivalence in geometric $K$-fixed points for all
$K\in \cK$ (one construction follows from the results below).  One then mimics the rest of the proof in the classical
case. 

It then follows that  $X\simeq X \sm \elr{\cK}$. 
Now we construct $\elr{\cK}$ out of basic cells $\sigma_K$ with $K
\in \cK$, and claim that $G/H_+\sm \sigma_K$ can be constructed from 
basic cells $\sigma_{K'}$ for $K'$ cotoral in $K$.
\end{remark}

The two natural notions of finiteness for rational spectra coincide.

\begin{lemma}
 A rational $G$-spectrum is constructed from finitely many 
basic cells $\sigma_{K, U_K}$ if and only if it is constructed from finitely
many classical cells $G/K_+$.
\end{lemma}

\begin{proof}
Since $\sigma_{K, U_K}$ is a retract of  $G/K_+$, a basic-finite complex is a
finite complex. The standard cell $G/K_+$ is a basic-finite complex
(it is built by basic cells using Lemma \ref{lem:basicbuildsclass},
and then $G/K_+$ is a retract of a finite basic complex using
smallness). Accordingly,  any
classical-finite complex is basic-finite.  
\end{proof}

\section{The classification of primes}
We wish to deduce the classification of prime ideals from our
classification of localizing tensor ideals. We pause to introduce the
underpinning  principle. 

\subsection{The Thomason Localization Theorem}
The reason that it is a viable strategy to first classify localizing
subcategories of infinite objects and then deduce classifications of
finite objects is Thomason's Localization Theorem. 

The proof is quite formal, but it is an extremely powerful general principle. It is well known
in various forms, but we want a version for tensor ideals in an
abstract setting; this is entirely in the spirit of \cite{Stevenson}
and I am grateful to G. Stevenson for the following proof showing that it follows
directly from the results there. 

\begin{thm}\label{thm:TLT}
 {\em (Thomason's Localization Theorem \cite{Neeman, Stevenson})}
(i) If $\C$ is a triangulated category generated by compact objects and 
$A$ is a set of small objects then 
$$\loc (A)\cap \C^c=\thick (A). $$

(ii) If $\C$ is a tensor triangulated category generated by compact objects and 
$A$ is a set of small objects then 
$$\loct (A)\cap \C^c=\thickt (A). $$
\end{thm}

\begin{remark}
We emphasize that in Part (ii), $\thickt(A)$ only permits tensoring with
compact objects, whilst $\loct(A)$ permits tensoring with arbitrary
objects. 
\end{remark}

\begin{proof}
Part (i) is \cite[2.1.3]{Neeman}. 

We deduce Part (ii) from Part (i). It suffices to show
$\loct(A)=\loc(\thickt (A))$ since then Part (i) shows
$$\loct(A)\cap \C^c=\loc (\thickt (A))\cap \C^c=\thickt(A). $$

It is clear that $\loct(A)\supseteq \loc(\thickt (A))$, so it suffices
to show that $\loc(\thickt(A))$ is an ideal. For this we calculate
$$\begin{array}{rcl}
\loc (\thickt(A))&=&\loc(\thick(\{ x\tensor a\st x\in \C^c, a\in
                     A\})) \mbox{ \cite[Lemma 3.9]{Stevenson}}\\
&=&\loc(\{ x\tensor a\st x\in \C^c, a\in A\})\\
&=&\loc( \C\tensor \loc(A)) \mbox{ \cite[Lemma 3.11]{Stevenson}}
\end{array}$$
Evidently $\C \tensor \loc (A)$ is closed under tensoring, so this is
an ideal by \cite[Lemma 3.9]{Stevenson}. 
\end{proof}

\subsection{The primes}

We finally want to prove that we have found all the primes.

\begin{lemma}
\label{lem:primenotint}
If $\wp $ is a prime and $\wp =\bigcap_{L\in A}\wp_L$ then $A$ has a
unique minimal element $L$ and $\wp=\wp_L$.
\end{lemma}

\begin{proof}
If not we can choose a subgroup $L$ in $A$ which is not
redundant. Thus 
$$\wp =\wp_L \cap \bigcap_{K\in A\setminus \{ L\} } \wp_K, $$
and since $\wp_L$ is not redundant, we may choose $X_L\in \wp_L\setminus \wp$ and 
$Y_L\in \bigcap_{K\in A\setminus \{ L\} } \wp_K\setminus \wp$. 
%These exist by Theorem \ref{thm:geomiso} since we can construct finite complexes with any geometric
%isotropy closed under passage to passage to subtoral subgroups, 
This contradicts the fact that $\wp$ is prime since
$X_L\sm Y_L \in \wp$ but $X_L\not \in \wp$ and $Y_L\not \in \wp$.
\end{proof}

\begin{thm}
\label{thm:allprimes}
The prime tensor ideals of compact $G$-spectra are precisely those of
the form 
$$\wp_K=\{ X\in \Gspecc \st \Phi^K X\simeq_1 0\}$$ 
for some closed subgroup $K$. The geometric isotropy  of $\wp_K$
consists of subgroups $H$ in which $K$ is not cotoral up to conjugacy 
$$\Ig (\wp_K)=All\setminus V(K). $$

\end{thm}
\begin{proof}
If $G$ is non-trivial the collection of contractible spectra is not
prime since if $F_1$ and $F_2$ are non-conjugate finite subgroups 
the smash product of $\sigma_{F_1}\sm \sigma_{F_2}$ is contractible.  

From Theorem \ref{thm:loctideals}, if $\wp$ is prime then 
$$\loct (\wp)=\sum_{K\in \Ig(\wp)}\GspecK.$$
By Thomason's Localization Theorem \ref{thm:TLT}, if  $Y\in \loct(\wp)$ is compact then $Y\in \thickt(\wp)=\wp$. Hence
$$\wp=\{ X \in \Gspec^c\st \Ig(X)\subseteq  \Ig(\wp)\}.$$

Now consider the complement of $\Ig (\wp)$. If it has a single
cotorally  minimal conjugacy class $K$ then $\wp=\wp_K$. 

If  $K_1, K_2$ are cotorally  minimal in the complement of $\Ig (\wp)$ then if $K_1$
and $K_2$ are not conjugate, we may choose open and closed neighbourhoods $U_i$ of $K_i$ in
$\Phi K_i$. Therefore  $\sigma_{K_1, U_1}\sm \sigma_{K_2, U_2}\in \wp$ but
$\sigma_{K_1, U_1}\not \in \wp$.
 \end{proof}

\section{Geometric isotropy of finite spectra}
\label{sec:giforfinite}
In this section we will give the classification of the collections of
subgroups occurring as the geometric isotropy of a finite
$G$-spectrum. Once again we apply Thomason's Localization Theorem to
deduce facts about thick tensor ideals of finite spectra from facts
about localizing tensor ideals of arbitrary spectra.

\subsection{Thick tensor ideals of finite spectra}
The geometric isotropy coincides with the support, and the point of
support is that it determines the thick tensor ideals.

\begin{thm}
If $X$ and $Y$ are finite $G$-spectra then 
if $\Ig(Y)\subseteq \Ig (X)$ then  $Y\in \thickt (X)$. 
\end{thm}

\begin{proof}
By Theorem \ref{thm:loctideals} $\loct(X)=\sum_{K\in \Ig (X)}\GspecK$, $Y \in 
\loct(X)$. By Thomason's Localization Theorem \ref{thm:TLT}, since $Y$
is  finite this means $Y\in \thickt(X)$. 
\end{proof}

The more interesting question is which sets occur as $\Ig (X)$ for finite
$G$-spectra $X$. Since the geometric isotropy is closed under cotoral
specialization by  Proposition \ref{prop:Igctclosed}, the set is
determined by its set $\ctmax(\Ig (X))$ of cotorally maximal
subgroups. 

For tori, the classification was given in \cite{spctq} (it will also follow
from Theorem \ref{thm:ctmaxcomp} below). 

\begin{example}
If $G$ is a torus then the sets $\Ig (X)$ for finite $X$ are precisely
those with finitely many cotorally maximal elements: the sets of the
form $\Lct (A)$ where $A$ is a finite set of cotorally unrelated
subgroups. 
\end{example}

However, for more general subgroups the statement will inevitably be
more complicated. 

\begin{example}
If $G=O(2)$ then for any $n$ the set  $U_n$ consisting of $O(2)$ and
the  dihedral subgroups of order $\geq 2n$ is clopen in $\Phi O(2)$ we see
$\ctmax(\Ig (\sigma_{O(2), U_n}))=(O(2), U_n)$.  
\end{example}

\subsection{Characterization of thick tensor ideals}

The following finiteness theorem is fundamental.

\begin{thm}
\label{thm:ctmaxcomp}
If $X$ is a finite spectrum then $\ctmax(\Ig (X))$ is an open compact
set in the $f$-topology.  All open, compact and cotorally undrelated
sets occur as $\ctmax(\Ig(X))$ for some finite $G$-spectrum $X$. 
\end{thm}

\begin{proof}
First we show that $\ctmax(\Ig (X))$ is open. 
Suppose that $K\in \Ig (X)$. If $K$ has no neighbourhood in $\Ig
(X)$ we can find a sequence of subgroups $L_i$ tending to $K$ with
$L_i\not \in \Ig (X)$, and by the Montgomery-Zippin theorem we may
suppose $L_i$ is a subgroup of $K$.  It suffices to view $X$ as a
$K$-spectrum, and by Illman's Theorem that too is finite. By the
Freudenthal  Suspension theorem, we may
suppose $X$ it is the suspension spectrum of a
finite $K$-CW-complex $Z$. There are finitely many non-$K$-fixed cells, 
so for $i$ sufficiently large, the subgroups $L_i$  are not conjugates
to subgroups of them and $\Phi^{L_i}Z=\Phi^KZ$. 

Next we show that $\ctmax(\Ig (X))$ is compact. 
For each $K\in \Ig (X)$ we choose a clopen neighbourhood $U_K$ of $K$
in $\Phi K$. Note that and since $\Ig (X)$ is open we may suppose $U_K\subseteq
\Ig (X)$. By Theorem \ref{thm:loctideals}, $X\in \loct \{ \sigma_{K, U_K}\st
K \in \Ig (X)\}$, and since $X$ is small, there are finitely many
subgroups $K_1, \ldots, K_N$ so that $X$ is finitely built by
$\sigma_{K_1, U_1}, \ldots , \sigma_{K_N, U_N}$. Considering the dual
process, $X$ can be converted to a point by using these same cells. 
Since $\ctmax(\sigma_{K, U_K})=(K, U_K)$ the only way that an element
of $\ctmax(\Ig (X))$ can be removed is if it lies in some $U_i$, so
that $U_1, \ldots , U_N$ is a finite cover of $\ctmax(\Ig (X))$.

Finally, we show that all compact, open and cotorally unrelated
subsets occur as the cotorally maximal elements in the geometric isotropy of a wedge of basic cells. 
All compact open subsets of $\sub (G)/G$ are finite unions of sets
$(K, U_K)$. We claim the same is true for those which are cotorally
unrelated. Such sets are realizable as the geometric isotropy of
$\bigvee_K \sigma_{K, U_K}$.  It remains to observe that if there are any cotoral
relations among the sets $(K,U_K)$ then we may replace them by a
smaller collection without cotoral relations.

Indeed, we need only note that for any $K, L$ the set
 $(L, U_L)\setminus \Lct (K, U_K)$ is a finite union of sets $(M,
 U_M)$. This follows since $\Phi L \cap \Lct (K, U_K)$ is an open 
set. This follows from Lemma \ref{lem:quottop}.
\end{proof}

Some may prefer the following reformulation. A thick tensor ideal is
called {\em finitely generated} if it is generated by a finite number
of  small spectra (or equivalently, by just one). 

\begin{cor}
\label{cor:geomiso}
The finitely generated thick tensor ideals of finite rational $G$-spectra are precisely
the fibres of geometric isotropy
$$\Ig : \mbox{finite-rational-$G$-spectra}\lra \mathcal{P}(\sub
(G)); $$
the image consists of collections of subgroups which are closed under cotoral 
specialization, and whose cotorally maximal parts are open and compact. \qqed
\end{cor}

\section{Zariski topologies}
\label{sec:top}
We have shown  that the map 
$$\sub (G)/G\stackrel{\cong} \lra \spc (\Gspectra)$$ 
associating a closed subgroup $K$ of $G$ with the Balmer prime $\wp_K$
is a bijection. Furthermore we have shown that this is an isomorphism
of posets in the sense that 
 $\wp_L\subseteq \wp_K$ if and only if a conjugate of $L$ is cotoral in
$K$. The purpose of this section is to show that topological
structures correspond. 

\subsection{The Zariski topology}
According to \cite[Definition 2.1]{Balmer1}, the Zariski topology on $\spc (\Gspectra)$ has closed sets
$$Z(\fX)=\{ \wp \st \wp \cap \fX=\emptyset\} =\bigcap_{X\in \fX}Z(X)$$
where $\fX$ runs through collections of finite rational
$G$-spectra. Since $X\in \wp_K$ if and only if $K\not \in \Ig (X)$,
under the correspondence between primes and subgroups,  we have 
$$Z(X)=\Ig(X), $$
and the topology is generated by the geometric isotropy of finite spectra. 

We may recover the poset structure on $\spc (\Gspectra)$ from the
topology since the closure of a prime consists of all primes contained
in it:
$$\overline{\{\wp\}}=\{ \mathfrak{q} \st \mathfrak{q} \subseteq
\wp\}. $$

We are rather used to the Zariski topology in the prime spectrum of a
Noetherian ring,  where a collection $A$ of maximal ideals is closed if and only if $A$
is finite (as forced by the fact that maximal ideals are
closed). However $\spc (\Gspectra)$ usually has the property that  many other subsets are
also closed. We will see that this can be viewed as saying that the
set of primes is a topological poset, with the poset encoding the
closures of points (a standard construction for $T_0$-spaces),  and
the much coarser topology viewing points at the ends of non-trivial morphisms as `far apart'.  

\begin{remark}
We may always view the Zariski topology on the Balmer spectrum 
as the $f$-topology on a poset. 

The closures of points form a poset since the Zariski topology 
is $T_0$. The $f$-topology is defined  to be generated by the closed sets $\ctmax{V}$ where $V$
runs through the generating closed sets $\supp(X)$. These are
generators since  $\ctmax(V\cup W)\supseteq \ctmax (V) \cup
\ctmax(W)$. 

We may recover the original Zariski topology by taking specialization closures. Again
the images of generators are generators since $\Lambda (A\cup
B)\supseteq \Lambda (A) \cup \Lambda (B)$. 
\end{remark}

\subsection{The Zariski $f$-topology}
We can also generate a Zariski $f$-topology (or $zf$-topology) on $\sub(G)/G$ by combining
the $f$-topology and the partial order. We define the sets  $\Lct
(V)$ to be  $zf$-closed   whenever $V$ an $f$-closed set, and then consider the
topology they generate. 

We note that the set $\Lct (A)$ consists of all sources of arrows
arriving in $A$, so that $\Lct (A)= s(t^{-1}(A))$. Since $s$ is an
open map, it is $f$-open if $A$ is $f$-open. 
%[Question: Is it necessarily $f$-closed?]

\subsection{The homeomorphism}

\begin{thm}
Associating a closed subgroup $K$ to a prime $\wp_K$ induces a homeomorphism
$$(\sub(G)/G, zf)\cong (\spc (\Gspectra), Zariski). $$
\end{thm}

\begin{proof}
The Zariski topology on $\sub(G)/G$ is generated by $\Lct (K,U_K)=\Ig
(\sigma_{K,U_K})$, so that the Zariski topology on $\spc(\Gspectra)$
is at least as fine. It remains to show that for any finite spectrum
$X$, the geometric isotropy $\Ig (X)$ is already in the topology
generated by the basic cells. This was the main part of Theorem
\ref{thm:ctmaxcomp}. 
\end{proof}

\section{Semifree $\protect \T$-spectra}
\label{sec:semifree}
The point of this section is to show that it is much harder to
classify thick subcategories than thick tensor-ideals. It will suffice
to look at semifree $\T$-spectra for the circle group $\T$ i.e., those $\T$-spectra with $\Ig
(X)\subseteq \{ 1, \T\}$. The model for these \cite{s1q}  is sufficiently simple
that we may be explicit. 

\subsection{The model of semifree $\protect\T$-spectra}
The  model $\cA_{sf}(\T)$ of semifree $\T$-spectra can be obtained from the model $\cA (\T)$ of all $\T$-spectra
by restriction, but it is easier to repeat the construction from scratch.  
In fact $\cA_{sf}(\T)$ is the abelian category of objects $\beta : N \lra \Q[c,c^{-1}]\tensor V$
where $N$ is a $\Q[c]$-module, $V$ is a graded $\Q$-vector space and $\beta$ is the $\Q[c]$
map inverting $c$. In effect, we have  the $\Q [c]$-module $N$, together with a chosen `basis'
$V$. Morphisms are commutative squares
$$\diagram 
M\rto^{\theta} \dto&N\dto\\
\Q[c,c^{-1}]\tensor U \rto^{1\tensor \phi}&\Q[c,c^{-1}]\tensor V
\enddiagram$$
The category $\cA_{sf}(\T)$ is of injective dimension 1, and the ring $\Q[c]$ is evenly graded, so every object of 
$d\cA_{sf}(\T)$ is formal, and we will identify semifree $\T$-spectra  $X$ (or objects of $d\cA_{sf}(\T)$) with their
image $\piA_*(X)$ in the abelian category $ \cA_{sf}(\T)$.

The fact we are talking about {\em ideals} is essential for Corollary \ref{cor:geomiso}. If we consider
semifree $G$-spectra when $G$ is the circle then there are just two thick tensor
ideals of finite spectra
\begin{itemize}
\item  free spectra (with geometric isotropy $1$, generated by $G_+$) 
\item  all spectra (with geometric isotropy $\{ 1, G\}$, generated by $S^0$). 
\end{itemize}
On the other hand, the thick subcategory generated by $S^0$ (without the ideal
property) does not contain $G_+$, and we will make it explicit. The classification of 
thick subcategories in general seems complicated, and we do not give a complete answer. 

\subsection{Wide spheres}
The small objects with $\beta$ injective are the objects $X=(\beta:
N\lra \Q[c,c^{-1}]\tensor V )$ with $\beta$ injective, $V$ finite
dimensional and $N$ bounded above; these objects are called {\em wide
  spheres} \cite{s1q}. 

We note that $\Q [c,c^{-1}]\tensor V$ is the same in each even degree
and the same in each odd degree. We therefore let
$$|V|_0=\bigoplus_{k}V_{2k} \mbox{ and }
|V|_1=\bigoplus_{k}V_{2k+1}. $$
We will fix  isomorphisms 
$$|V|_0\cong (\Q [c, c^{-1}]\tensor V)_0\mbox{ and }|V|_1\cong (\Q [c,
c^{-1}]\tensor V)_1, $$
and then use multiplication by powers of $c$ to identify other graded
pieces of $\Q [c,c^{-1}] \tensor V$ with the appropriate one. 

We will want to think of stepping down the degrees in steps of 2, so
we take
$$|V|_{\geq 2k}=\bigoplus_{n\geq k}V_{2n} \mbox{ and }
|V|_{\geq 2k+1}=\bigoplus_{n\geq k}V_{2n+1}$$
for the parts of $V$ above a certain point, but in the same parity. 

Similarly, we move $N_{2k}$ into degree 0 by multiplication by $c^k$:
$$\Nb_{2k}:=c^kN_{2k}\subseteq |V|_0 \mbox{ and } \Nb_{2k+1}:=c^kN_{2k+1}\subseteq |V|_1.$$
Having established the framework, we will restrict the discussion to
the even part, leaving the reader to make the odd case explicit. 

 If $X$ is nonzero in even degrees, since
$X$   is small there is a highest degree $2a-2$ in which  $N$ is non-zero, and since
$N[1/c]=\Q[c, c^{-1}]\tensor V$ there is highest degree $2b$ in which
$N$ coincides with $|V|_0$. Accordingly, we have a finite filtration 
$$0=\Nb_{2a} \subseteq \Nb_{2a-2} \subseteq \cdots \subseteq \Nb_4\subseteq \Nb_2 \subseteq \cdots
\subseteq \Nb_{2b}=|V|_0. $$

We wish to consider two increasing filtrations on $|V|_0$
$$\cdots \subseteq |V|_{\geq 2k+2}\subseteq |V|_{\geq 2k}\subseteq |V|_{\geq 2k-2}\subseteq
\cdots \subseteq |V|_0$$
and 
$$\cdots \subseteq \Nb_{ 2k+2}\subseteq \Nb_{\geq 2k}\subseteq \Nb_{\geq 2k-2}\subseteq
\cdots \subseteq  |V|_0. $$

\subsection{Two conditions on wide spheres}
In crude terms, we will show the thick subcategory generated by $S^0$ consists of
objects so that (a) the dimensions of the spaces in the $V$- and $N$-filtrations agree and (b)
the $V$ filtration is slower than the $cN$ filtration. The purpose of
this subsection is to introduce the two conditions. 

\begin{cond}
\label{cond:thickS}
We say that a wide sphere is {\em untwisted} if it satisfies the
following two conditions
\begin{enumerate}
\item $\dim (\Nb_{i})=\dim (|V|_{\geq i})$ for all $i$
\item $V\cap cN =0$
\end{enumerate}
\end{cond}

We will be showing that these characterize the thick subcategory
generated by $S^0$. We must at least show that the conditions are
inherited by retracts, and this verification will lead us to some
useful introductory discussion.

\begin{lemma}
\label{lem:condretract}
Condition \ref{cond:thickS} is closed under passage to retracts. 
\end{lemma}

\begin{proof}
It is immediate that Conditon \ref{cond:thickS} (ii) is inherited by
retracts. We also note that Conditon \ref{cond:thickS} (ii) implies one of the
inequalities  for Conditon \ref{cond:thickS} (i):
$$\dim (\Nb_{i})\geq \dim (|V|_{\geq i}). $$

Now suppose $X=X'\oplus X''$ and that $X$ satisfies Condition
\ref{cond:thickS}. As observed already,  $X'$ and $X''$ both satisfy the second condition, and hence both 
satisfy the first condition with $=$ replaced by $\geq $. With lower
case letters denoting dimensions of vector spaces (for example
$n_a=\dim (N_a)$), this means we 
have a pair of  increasing sequences 
$$0=n'_a, n'_{a-2},n'_{a-4},  \cdots \mbox{ and }
0=v'_{\geq a}, v'_{\geq a-2},v'_{\geq a-4} \cdots $$
reaching $v'$
and a pair of increasing sequences
$$0=n''_a, n''_{a-2},n''_{a-4},  \cdots \mbox{ and }
0=v''_{\geq a}, v''_{\geq a-2},v''_{\geq a-4}\cdots $$
reaching $v''$. 
Since $X$ satisfies Condition \ref{cond:thickS}(i), the sum of the
first pair and the second pair give two equal sequences (i.e.,  the
sequence $n'_i+n''_i=n_i$ and the sequence $v'_{\geq i}+v''_{\geq
  i}=v_{\geq  i}$ are equal). Thus if one pair deviates from equality in the positive
direction, the other deviates in the negative direction. Since
Condition \ref{cond:thickS}(ii) shows there is no negative deviation,
we must have equality for both pairs throughout. 
\end{proof}

It is useful to be able to consider the changes of dimension and form
the  generating function. In fact to any wide sphere, we may associate to it two Laurent polynomials
\begin{itemize}
\item The geometric $T$-fixed point polynomial
$$p_{\T}(t)=\sum_i\dim_{\Q} (V_i)t^i$$
\item The 1-Borel jump polynomal 
$$p_1(t)=\sum_i\dim_{\Q} (N_i/cN_{i+2})t^i$$
\end{itemize}
Condition \ref{cond:thickS}(i) is then equivalent to the condition 
$$p_{\T}(t)=p_1(t). $$

\begin{remark}
We note that Condition \ref{cond:thickS}(i)  is not closed under passage to
retracts. Indeed, $S^z\vee S^{2-z}$ satisfies the first condition with 
$p_1(t)=p_\T(t)=t^2+1. $
However  $S^z$ (with $p_\T(t)=1$ and $p_1(T)=t^2$)
 and $S^{2-z}$ (with $p_\T(t)=t^2$ and $p_1(T)=1$) do not.
  \end{remark}

\subsection{Attaching a $\protect \T$-fixed sphere}

To start with,   $S^0=(\Q [c]\lra \Q [c, c^{-1}]\tensor \Q)$ and then
direct sums of these model wedges of $\T$-fixed spheres with  $N=\Q [c]\tensor V$, and of course
it is easy to see that Condition \ref{cond:thickS}  holds for these. 

However $N$ does not always sit 
so simply inside $\Q[c,c^{-1}]\tensor V$  for the objects built from
$S^0$. We may see this in a simple example. 

\begin{example}
Up to equivalence there are precisely three wide spheres with
$p_1(t)=p_\T(t)=1+t^2$. Evidently in all cases $V=\Q\oplus \Sigma^2 \Q$, $\Nb_{2k}=0 $ for $k\geq 4$ and $\Nb_{2k}=|V|$ for $k\leq 0$. The
only question is how the 1-dimensional space $\Nb_2$ sits inside $|V|=
V_0\oplus V_2$. The three cases are $N_2=V_0$ (which is $S^z\vee
S^{2-z}$),  $N_2=V_2$ (which is $S^0\vee S^2$), and the third case
(giving just one isomorphism type)  in which $N_2$ is a 1-dimensional
subspace not equal to $V_0$ or $V_1$.

We note that the third example is the mapping cone $M_f$ for any
non-trivial map $f: S^1\lra S^0$ (in the semifree category,  there is
only one up to multiplication by a non-zero scalar). In this case up to
isomorphism,  $N_2$ is
generated by  $c^{-1}\tensor \iota_0+c^0\tensor \iota_2$. 

We observe then that the second and third of these three are in
$\thick (S^0)$, and we see that the first does not satisfy Condition
\ref{cond:thickS} (ii). 
\end{example}

\begin{lemma}
\label{lem:condaddcell}
Given a cofibre sequence, 
$$S^n \stackrel{f}\lra X\lra Y,  $$
if $X$ is a wide sphere then so is $Y$ and if $X$ in addition satisfies Condition \ref{cond:thickS} then so does $Y$.
\end{lemma}

\begin{proof}
Suppose first that $X$ is entirely in one parity. Without loss of
generality, we may suppose $X$ is in even degrees. 

If $n$ is odd then $\piA_*(S^n)$ is purely odd and we have a short
exact sequence
$$0\lra \piA_*(X)\lra \piA_*(Y)\lra  \piA_*(S^{n+1})\lra 0. $$
It follows that $Y$ is a wide sphere (i.e., the basing map is
injective). 
The condition on dimensions is immediate, since this is split as
vector spaces. For the second condition, we know that any element 
$(v,\lambda \iota ) \in V_Y\cap cN_Y$ with $v\in V_X$ must have $\lambda
\neq 0$ since $X$ satisfies the condition. However $\lambda \iota \not
\in c\Q [c]$. Altogether,   $Y$ satisfies
Condition \ref{cond:thickS}.  

Alternatively,  suppose $n$ is even. To calculate $[S^n, X]^\T_*$ we
take an injective resolution of $X$. We argue that  this takes the form
$$ 0\lra X \lra e(V) \lra f(\Sigma^2V\tensor k[c]^{\vee})\lra 0. $$
To start with,  since $X$ is a wide sphere,  $X$
embeds in $e(V)$. The cokernel is zero at $\T$ and hence of the form
$f(T)$ for some torsion $\Q [c]$-module $T$. At  $1$ the cokernel is $(\Q
[c,c^{-1}]\tensor V)/N$; since this is  divisible 
it is a sum of copies of $\Q[c]^{\vee}$. Finally, in view of Condition
\ref{cond:thickS} (i) $T=\Sigma^2 V\tensor \Q[c]^{\vee}$ as claimed.

Now we may use the Adams spectral sequence to see that $[S^n,
X]^\T_0=V_n$. If the original  map $f$ is trivial, then 
 $\piA_*(Y)=\piA_*(X)\oplus \piA_*(S^{n+1})$, and the result is 
again clear. Otherwise we have a diagram
$$\diagram 
S^n \rto^f &X\\
\Sigma^n\Q[c]\rto^{\theta} \dto & N\dto \\
\Q[c,c^{-1}]\tensor \Sigma^n\Q\rto^{1\tensor \phi} & \Q [c,c^{-1}]\tensor V 
\enddiagram$$
This shows that since $\phi $ is mono then also $\theta$ is mono and
furthermore, by Condtion \ref{cond:thickS}(ii), $\theta$ is the
inclusion of a summand. It follows that the map $f$ is split. Indeed,
splittings of $\phi$ are given by codimension 1 free summands $N'$ of
$N$. This automatically has $\Nb'$ of codimension 1 in $|V|$. 
We need only choose $N'$  so that $\Nb'$ avoids $\theta
(\Sigma^n\Q)$. This gives a compatible splitting of $\phi$. 
It follows that in fact $Y$ is a retract of $X$, and the result
follows from Lemma \ref{lem:condretract}. 

Finally if $X$ has components in both even and odd degrees, then $X\simeq X_{ev}\vee X_{od}$ and we may argue as follows. Without
loss of generality we suppose $n$ is even. 
If $f$ maps purely into $X_{ev}$ or purely into $X_{od}$ the other factor is irrelevant and the above argument deals with this case. 
Otherwise $f$ has components mapping into both $X_{ev}$ and
$X_{od}$. The above argument shows that $\piA_*(Y)$ is a retract in
even degrees and it is unaltered in odd degrees.
\end{proof}

\subsection{Spectra built from $\protect \T$-fixed spheres}
We have now done the main work and can identify the thick subcategory
generated by $S^0$.

\begin{cor}
\label{cor:thickS}
The thick subcategory generated by $S^0$ consists of wide spheres
satisfying Condition \ref{cond:thickS}. 
\end{cor}

\begin{proof}
First, we observe that the thick subcategory $\thick (S^0)$ can be
constructed
by alternating the attachment of $\T$-fixed spheres and taking
retracts; the fact that any element of $\thick (S^0)$ is a wide sphere
satisfying Condition \ref{cond:thickS} then follows from Lemmas
\ref{lem:condretract} and \ref{lem:condaddcell}. The point is that 
we must show that if  we construct $Z$ using a cofibre sequence $X\lra
Y \lra Z$ with $X,Y$ in the thick subcategory then $Z$ may be
constructed from $X$ by using the two processes. Formally, we are
applying induction on the number of cells, so we may suppose $Y$ is
constructed from the two processes. If $X$ is formed by attaching spheres, we may form $Z$ from
$Y$ by attaching the corresponding spheres. If $X$ is a retract of
$X'$ formed from spheres  then $f: X\lra Y$ extends over $X'=X\vee
X''$ by using $0$ on the second factor and  then $Z$ is a retract of $Z'$. 

Now  we show that any wide sphere satisfying Condition \ref{cond:thickS} is in the
thick subcategory generated by $S^0$. 
We argue by induction on the dimension of $|V|$. The result is obvious if $V=0$. 
Suppose that $X$ is a wide sphere satisfying the given condition and
that the result is proved when the geometric $\T$-fixed points have lower dimension. 

Note that if $t^n$ is the smallest degree in which $p_\T(t) $ is non-zero we may choose a vector $v\in V_n\setminus \Nb_{n+2}$. Accordingly
$X$  has a direct summand $\Q[c]\tensor v\lra \Q [c,c^{-1}]\tensor v$, which corresponds to a map 
$S^n\lra X$. Since $v\not \in \Nb_{n+2}$, the  quotient $Y$ again has injective basing map and obviously satisfies the polynomial condition. Since $n$ is the smallest
degree in which $V$ is non-zero, $v$,  the direct summand
$\Q \cdot v$ may be removed from $|V|$ without affecting the filtration condition. By induction $Y\in \thick (S^0)$, and hence $X\in \thick (S^0)$ as  required. 
\end{proof}

\subsection{Spectra built from representation spheres}
Since smashing with any sphere $S^{kz}$ is invertible, this allows us
to deduce the thick subcategory generated by any sphere. 
\begin{cor}
The thick subcatgory generated by $S^{kz}$ consists of 
wide spheres which are $k$-twisted in the sense that 
\begin{enumerate}
\item  $p_1(t)=t^{-2k}p_\T(t)$.
\item $V\cap c^{k+1}N=0$
\end{enumerate}
\end{cor}

\begin{proof}
This is immediate from Corollary \ref{cor:thickS} and  the observations
$$V(S^{kz}\sm X)=V(X) \mbox{ and } N(S^{kz}\sm X)=c^{-k}N(X).$$
\end{proof}

\end{document}